\documentclass[12pt]{amsproc}

\usepackage[top=1in,bottom=1in,left=1.5in,right=1.5in]{geometry}
\usepackage{amsfonts} %\mathbb{R}, etc.
\usepackage{amsmath}
\usepackage{amssymb}
\usepackage{amsthm}
\usepackage{mathrsfs}
\usepackage{setspace}
\usepackage{subfigure}
\usepackage{url}
\usepackage{booktabs}
\usepackage{ifthen}
\usepackage{tikz}
\usetikzlibrary{calc}
\usepackage{enumerate}

\newcommand{\C}{\mathscr{C}}

\newtheorem{thm}{Theorem}[section]

\newtheorem{obs}[thm]{Observation}
\newtheorem{claim}[thm]{Claim}

\newtheorem{defn}[thm]{Definition}

\numberwithin{equation}{section}

\begin{document}

\title{Vizing's conjecture: a two-thirds bound for claw-free graphs}

\author{Elliot Krop}
\address{Elliot Krop (\tt elliotkrop@clayton.edu)}
\address{Department of Mathematics, Clayton State University}
\date{\today}

\begin {abstract}
We show that for any claw-free graph $G$ and any graph $H$, $\gamma(G\square H)\geq \frac{2}{3}\gamma(G)\gamma(H)$, where $\gamma(G)$ is the domination number of $G$.
\\[\baselineskip] 2010 Mathematics Subject
      Classification: 05C69
\\[\baselineskip]
      Keywords: Domination number, Cartesian product of graphs, Vizing's conjecture
\end {abstract}

\maketitle
 
 \section{Introduction}
 For basic graph theoretic notation and definitions see Diestel~\cite{Diest}. All graphs $G(V,E)$ are finite, simple, connected, undirected graphs with vertex set $V$ and edge set $E$. We may refer to the vertex set and edge set of $G$ as $V(G)$ and $E(G)$, respectively.
 
 For any graph $G=(V,E)$, a subset $S\subseteq V$ \emph{dominates} $G$ if $N[S]=V(G)$. The minimum cardinality of $S \subseteq V$ dominating $G$ is called the \emph{domination number} of $G$ and is denoted $\gamma(G)$. We call a dominating set that realizes the domination number a $\gamma$-set.

 \begin{defn}
 The \emph{Cartesian product} of two graphs $G_1(V_1,E_1)$ and $G_2(V_2,E_2)$, denoted by $G_1 \square G_2$, is a graph with vertex set $V_1 \times V_2$ and edge set $E(G_1 \square G_2) = \{((u_1,v_1),(u_2,v_2)) : v_1=v_2 \mbox{ and } (u_1,u_2) \in E_1, \mbox{ or } u_1 = u_2 \mbox{ and } (v_1,v_2) \in E_2\}$.
\end{defn}

In 1963, Vadim G. Vizing posed his now famous conjecture:
For any pair of graphs $G$ and $H$,
\begin{align}
\gamma(G \square H) \geq \gamma(G)\gamma(H).\label{V}
\end{align}

The statement is known for graphs with domination number two \cite{BDGHHKR} and three \cite{Sun}. Recently, Bo\v{s}tjan Bre\v{s}ar produced a clear and concise new proof of the result for graphs with domination number three \cite{B}.

The best current bound for the conjectured inequality was shown in 2010 by Suen and Tarr \cite{ST}, 
\begin{align*}
\gamma(G \square H) \geq \frac{1}{2}\gamma(G)\gamma(H)+\frac{1}{2}\min\{\gamma(G),\gamma(H)\} %\label{SuenTarr}
\end{align*}

In the survey \cite{BDGHHKR}, the authors proved a slightly better bound for claw-free graphs, showing that for any claw-free graph $G$ and any graph $H$, $\gamma(G\square H)\geq \frac{1}{2}\gamma(G)(\gamma(H)+1)$.

\medskip

In this note we apply the Contractor-Krop overcount technique \cite{CK} to the method of Bre\v{s}ar \cite{B} to show that for any claw-free graph $G$ and any graph $H$, $\gamma(G\square H)\geq \frac{2}{3}\gamma(G)\gamma(H)$.

\subsection{Notation}

A graph $G$ is \emph{claw-free} if $G$ contains no induced $K_{1,3}$ subgraph. 

Let $\Gamma=\{v_1,\dots, v_k\}$ a minimum dominating set of $G$ and for any $i\in [k]$, define the set of \emph{private neighbors} for $v_i$, $P_i=\big\{v\in V(G)-\Gamma: N(v)\cap \Gamma = \{v_i\}\big\}$. For $S\subseteq [k]$, $|S|\geq 2$, we define the \emph{shared neighbors} of $\{v_i:i\in S\}$ as $P_S=\big\{v\in V(G)-\Gamma: N(v)\cap \Gamma=\{v_i: i\in S\}\big\}$.

For any $S\subseteq [k]$, say $S=\{i_1,\dots, i_s\}$ where $s\geq 2$, we will usually write $P_S$ as $P_{i_1,\dots, i_s}$.

For $i\in [k]$, let $Q_i=\{v_i\} \cup P_i$. We call $\mathcal{Q}=\{Q_1,\dots, Q_k\}$ the \emph{cells} of $G$. For any $I\subseteq [k]$, we write $Q_I=\bigcup_{i\in I}Q_i$ and call $\C(Q_I)=Q_I\cup\bigcup_{S\subseteq I}P_{S}$ the \emph{chamber} of $Q_I$. We may write this as $\C_{I}$.

In Figure \ref{chamber} below, the black vertices are in the minimum dominating set. The chamber of $Q_{1,2,3}$ is composed of the black and gray vertices.

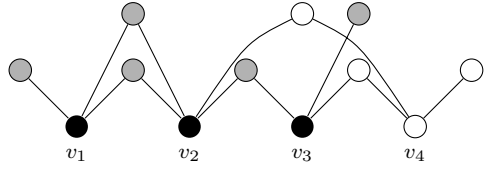
\begin{figure}[ht]
\begin{center}
\begin{tikzpicture}[scale=1.5]
\tikzstyle{vert}=[circle,fill=black,inner sep=3pt]
\tikzstyle{overt}=[circle, draw, inner sep=3pt, minimum size=6pt]
\tikzstyle{rvert}=[circle,draw,fill=black!30,inner sep=3pt]

\node[vert, label=below:\tiny{$v_1$}] (v1) at (1,1) {};
\node[vert, label=below:\tiny{$v_2$}] (v2) at (2,1) {}; 
\node[vert, label=below:\tiny{$v_3$}] (v3) at (3,1) {};
\node[overt, label=below:\tiny{$v_4$}] (v4) at (4,1) {};

\node[rvert] (u1) at (.5,1.5) {};
\node[rvert] (u2) at (1.5,1.5) {};
\node[rvert] (u3) at (1.5,2) {};
\node[rvert] (u4) at (2.5,1.5) {};
\node[rvert] (u5) at (3.5,2) {};
\node[overt] (u6) at (3.5,1.5) {};
\node[overt] (u7) at (4.5,1.5) {};
\node[overt] (u8) at (3,2) {};

\draw[color=black] 
   (u1)--(v1)--(u2)--(v2) (v1)--(u3)--(v2) (v2)--(u4)--(v3) (v3)--(u5) (v3)--(u6)--(v4) (v4)--(u7) (v2).. controls (2.5,1.75) .. (u8)..controls (3.5,1.75)..(v4);

\end{tikzpicture}
\caption{The Chamber of $Q_{1,2,3}$}
\label{chamber}
\end{center}
\end {figure}

For a vertex $h\in V(H)$, the \emph{$G$-fiber} of $h$, $G^h$, is the subgraph of $G\square H$ induced by $\{(g,h):g\in V(G)\}$. %Similarly, for a vertex $g\in V(G)$, the $H$-fiber, $H^g$, is the subgraph of $G\square H$ induced by $\{(g,h):h\in V(H)\}$.

For a minimum dominating set $D$ of $G\square H$, we define $D^h=D\cap G^h$. Likewise, for any set $S\subseteq [k]$, $P_S^h=P_S \times \{h\}$, and for $i\in [k]$, $Q_i^h=Q_i\times \{h\}$. By $v_i^h$ we mean the vertex $(v_i,h)$. For any $I^h\subseteq [k]$, where $I^h$ represents the indices of some cells in $G$-fiber $G^h$, we write $\C_{I^h}$ to mean the chamber of $Q^h_{I^h}$, that is, the set $\bigcup_{i\in I^h}Q_i\cup\bigcup_{S\subseteq I^h}P^h_{S}$.

%We may write $\{v_i:i\in I^h\}$ for $\{v^h_i:i\in I^h\}$ when it is clear from context that we are talking about vertices of $G\square H$ and not vertices of $G$.

For ease of reference, assume that our representation of $G\square H$ is with $G$ on the $x$-axis and $H$ on the $y$-axis.

Any vertex $(v,h)=v^h\in G^h$ is \emph{vertically dominated} by $D$ if $(\{v\}\times N_H[h])\cap D \neq \emptyset$. Vertices that are not vertically dominated are called \emph{vertically undominated}.
For $i\in [k]$ and $h\in V(H)$, we say that the cell $Q_i^h$ is \emph{vertically dominated} if $(Q_i\times N_H[h])\cap D\neq\emptyset$. A cell which is not vertically dominated is \emph{vertically undominated}. Note that all vertices of a vertically undominated cell $Q_i^h$ are dominated by vertices $(u,h)=u^h\in D^h$.

An \emph{independent dominating set} of a graph $G$ is a set of independent (pairwise mutually non-adjacent) vertices which dominate $G$. The size of a smallest independent dominating set of $G$ is denoted by $i(G)$.

%Given a set of vertices $\{v_1,\dots, v_{k_1}, u_1, \dots, u_{k_2}\}$ we say the set \linebreak $\{v_1,\dots, v_{k_1}, u_1^-, \dots, u_{k_2}^-\}$ dominates $S\subseteq V(G)$ if $S\subseteq N[v_1,\dots, v_{k_1}]$\linebreak $\cup \{u_1,\dots, u_{k_2}\}$.

%In our argument, we label vertices of a minimum dominating set $D$ of $G\square H$, by labels from $[k]$ so that for any $i\in [k]$, projecting the vertices labeled by $i$ onto $H$ produces a dominating set of $H$. We call a vertex $(x, h)\in D^h$ with the single label $i$, \emph{free}, if there exists another vertex $(y, h) \in D^h$, which is given the label $i$. A vertex $(x, h)\in P_{i,j}^h$ is called \emph{adjustable}, if the cells $Q_i^h$ and $Q_j^h$ are vertically dominated.

%\begin{defn}
%If a graph $G$ with cells $\mathcal{Q}$ contains $i$ vertices that dominate $j$ cells for some non-negative integers $i,j$, we call the set of vertices a \emph{$(i,j)$-control set}. The difference $i-j$ is called the \emph{discrepancy}.
%\end{defn}

%Note that for any control set, discrepancy must be non-negative, else a dominating set could be produced with fewer than $\gamma(G)$ vertices.

\section{Claw-free graphs}

We begin with the fundamental result on the domination of claw-free graphs.

\begin{thm}[Allan and Laskar \cite{AL}]\label{AL}
If $G$ is claw-free, then $i(G)=\gamma(G)$.
\end{thm}

The following fact follows from the definition of claw-free graphs.

\begin{obs}\label{private}
For any claw-free graph $G$ with minimum independent dominating set $\{v_1,\dots,v_k\}$, for any $S\subseteq [k]$ with $|S|\geq 3$, $P_S=\emptyset$.
\end{obs}

Our argument, like that of Bartsalkin and German \cite{BG}, relies on labeling the vertices of a minimum dominating set, $D$, of $G\square H$ with labels that contain integers from $\{1,\dots, \gamma(G)\}$. Labels may be sets of integers of size one or  pairs of distinct integers. We show that every set of labels containing a fixed integer is at least of size $\gamma(H)$. We then control the overcount of vertices by applying the method of Contractor and Krop \cite{CK}. This is done by first applying a series of three labelings of the vertices of $D$. Labels may contain one or two integers and in each successive labeling, we reduce the number of labels with two integer while at the same time maintaining the property that vertices with labels that contain a fixed integer, when projected onto $H$, form a dominating set of $H$.  

In particular, Labeling $1$ gives a singleton label to vertices of $D$ which can be projected onto a fixed dominating set or the private neighbors of the dominating set of $G$. Other vertices of $D$ are given a paired label. Labeling $2$ reduces the number of paired labels that interact with each other in different $G$-fibers while Labeling $3$ reduces the number of paired labels that interact with each other in the same $G$-fiber. 

The resulting relabeled set $D$ satisfies the property that every $G$-fiber with a certain number of vertices labeled by two integers must contain at least as many vertices labeled by one integer. This allows us to show the claimed lower bound on $|D|$.  

\begin{thm}\label{claw-free}
For any claw-free graph $G$ and any graph $H$, \[\gamma(G\square H)\geq \frac{2}{3}\gamma(G)\gamma(H).\]
\end{thm}

\begin{proof}
Let $G$ be a claw-free graph and $H$ any graph. We apply Theorem \ref{AL} and consider a minimum independent dominating set of $G$, $\Gamma=\{v_1,\dots, v_k\}$. Let $D$ be a minimum dominating set of $G\square H$. 

Our proof is composed of a series of increasingly refining labelings of the vertices of $D$. In all instances, for any $i,j\in [k]$ and $h\in V(H)$, if $v\in P^h_{i,j}$, then $v$ may be labeled by singleton labels $i,j,$ or paired labels $(i,j)$. %We call labelings that follow this property \emph{faithful}. 

Our goal is to reduce the number of paired labels as much as possible.

For any $h\in V(H)$, suppose the fiber $G^h$ contains $\ell_h(=\ell)$ vertically undominated cells $U=\big\{Q_{i_1}^h,\dots, Q_{i_{\ell}}^h\big\}$ for some $0\leq \ell \leq k$. We set $I^h=\{i_1,\dots, i_{\ell}\}$. 
\medskip

We apply the procedure \emph{Labeling 1} to the vertices of $D$. If a vertex of $D^h$ for any $h\in H$, is in $Q_{j_1}^h$ for $1\leq j_1 \leq k$, then we label that vertex by $j_1$. If $v\in D^h$ is a shared neighbor of some subset of $\{v_{j_1}:j_1\in I^h\}$, then by Observation \ref{private}, it is a member of $P^h_{j_1,j_2}$ for some $j_1,j_2\in I^h$, and we label $v$ by the pair of labels $(j_1,j_2)$. If $v$ is a member of $D\cap P^h_{j_1,j_2}$ for $i\in I^h$ and $j_2\in [k]-I^h$, then we label $v$ by $j_1$. If $v$ is a member of $D\cap P^h_{j_1,j_2}$ for $j_1,j_2 \in [k]-I^h$, then we label $v$ by either $j_1$ or $j_2$ arbitrarily. This completes Labeling $1$.

After Labeling $1$, all vertices of $D$ have a singleton label or a paired label.
\medskip

We relabel the vertices of $D$, doing so in $D^h$ for fixed $h\in H$, stepwise, until we exhaust every $h\in H$. This procedure, which we call \emph{Labeling 2}, is described next. 

Suppose $v^h\in P^h_{j_1,j_2}\cap D$ for some $j_1,j_2\in I^h, h\in V(H)$, and there exists $y^{h'}\in P^{h'}_{j_1,j_2}\cap D$ for $h' \in V(H), h'\neq h$. The vertex $y^{h'}$ may be labeled by a singleton or paired label, whether Labeling $2$ had been performed on $D^{h'}$ or not.

Suppose that $y^{h'}$ is labeled by a singleton label, say $j_1$. Remove the paired label $(j_1,j_2)$ from $v^h$ and relabel $v^h$ by $j_2$.

Suppose $y^{h'}$ is labeled by the paired label $(j_1,j_2)$. Remove the paired label $(j_1,j_2)$ from $v^h$ and then relabel $v^h$ arbitrarily by one of the singleton labels $j_1$ or $j_2$, and then relabel $y^{h'}$ by the other singleton label. This completes Labeling $2$.

After Labeling $2$, a vertex $v^h$ of $D$ may have a paired labels $(j_1,j_2)$ if $v^h\in P^h_{j_1,j_2}$ and for any $h'\in N_H(h)$, $D^{h'}\cap P^{h'}_{j_1,j_2}=\emptyset$.

We show an example of some labels after Labeling $2$ in Figure $2$.

\begin{center}
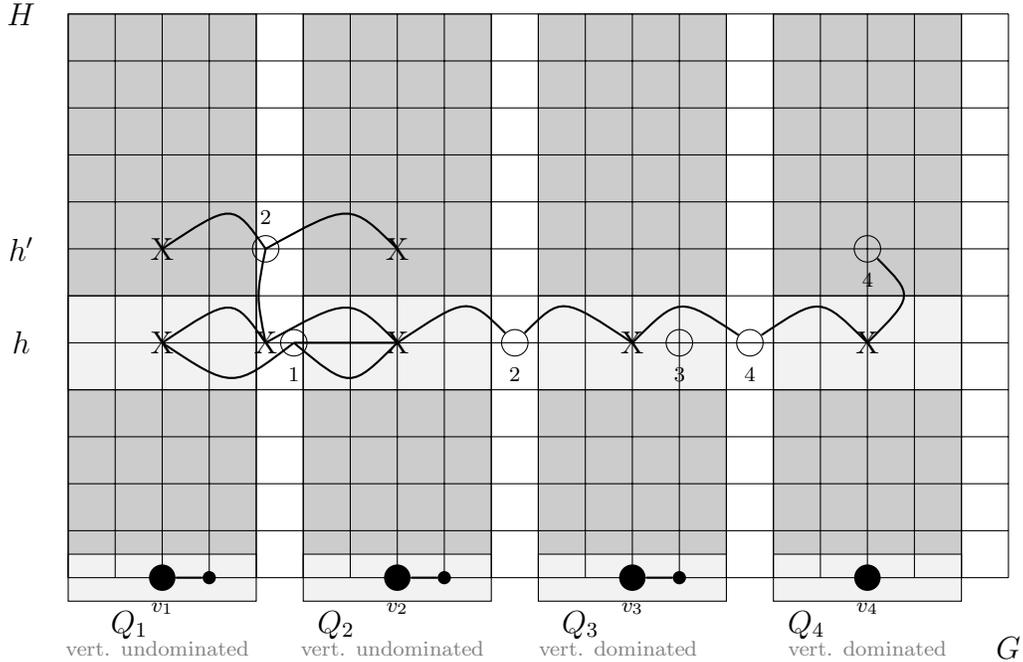
\begin{figure}[htb!]
 \begin{tikzpicture}[scale=1.25]
 
	 \tikzstyle{vertex}=[circle,minimum size=10pt,inner sep=0pt]
	 \tikzstyle{domvertex}=[vertex, fill=black]
	 \tikzstyle{pvert}=[circle,minimum size=5pt,inner sep=0pt,fill=black]
	 \tikzstyle{selected vertex} = [vertex, fill=red]
	 \tikzstyle{selected edge} = [draw,thick,-,red]
	 \tikzstyle{blue edge} = [draw,thick,-,blue]
	 \tikzstyle{edge} = [draw,thick,-,black]

	%\draw [help lines] (0,0) grid (10,6);

	\node at (10,-.75){$G$};
	\node at (-.5, 6){$H$};
	\node at (-.5, 2.5){$h$};
	\node at (-.5, 3.5){$h'$};

	\filldraw[fill=gray!40!white] (0,0) rectangle (2,6);

	\filldraw[fill=gray!40!white] (2.5,0) rectangle (4.5,6);

	\filldraw[fill=gray!40!white] (5,0) rectangle (7,6);

	\filldraw[fill=gray!40!white] (7.5,0) rectangle (9.5,6);

\filldraw[fill=gray!10!white] (0,2) rectangle (2,3);
\filldraw[fill=gray!10!white] (2.5,2) rectangle (4.5,3);
\filldraw[fill=gray!10!white] (5,2) rectangle (7,3);
\filldraw[fill=gray!10!white] (7.5,2) rectangle (9.5,3);

\filldraw[fill=gray!10!white] (0,-.25) rectangle (2,.25);
\filldraw[fill=gray!10!white] (2.5,-.25) rectangle (4.5,.25);
\filldraw[fill=gray!10!white] (5,-.25) rectangle (7,.25);
\filldraw[fill=gray!10!white] (7.5,-.25) rectangle (9.5,.25);

	\draw[step=.5,black,thin] (0,0) grid (10,6);

\node at (.65,-.5){$Q_1$};
\node at (2.85,-.5){$Q_2$ };
\node at (5.45,-.5){$Q_3$ };
\node at (7.85,-.5){$Q_4$ };

\node at (.95,-.75){\tiny{\color{gray} vert. undominated}};
\node at (3.45,-.75){\tiny{\color{gray} vert. undominated}};
\node at (5.85,-.75){\tiny{\color{gray} vert. dominated}};
\node at (8.5,-.75){\tiny{\color{gray} vert. dominated}};

\node[domvertex](v1) at (1,0)[label=below:\tiny{$v_1$}]{};
\node[domvertex](v2) at (3.5,0)[label=below:\tiny{$v_2$}]{};
\node[domvertex](v3) at (6,0)[label=below:\tiny{$v_3$}]{};
\node[domvertex](v4) at (8.5,0)[label=below:\tiny{$v_4$}]{};

\node[pvert](p1) at (1.5,0)[]{};
\node[pvert](p2) at (4,0)[]{};
\node[pvert](p3) at (6.5,0)[]{};

\draw[edge] (p1) -- (v1);
\draw[edge] (p2) -- (v2);
\draw[edge] (p3) -- (v3);

%\uncover<1>{
%\node[vertex](y) at (2.1,3.5)[label=above: \tiny{$y(1,2)$}][circle,draw=black]{};
%\node[vertex](x) at (2.4,2.5)[label=below right: \tiny{$x(1,2)$}][circle,draw=black]{};
%}

%\node[vertex](u2) at (2.1,2.5)[label=below: {}]{X};
%\draw[edge](u2)..controls (2,3) and (3,3)..(3.5,2.5);
\node at (2.1,2.5){X};
\node at (1,2.5){X};
\node at (3.5,2.5){X};

\draw[edge](2.1,2.5)..controls (3,3)..(3.5,2.5);
\draw[edge](1,2.5)..controls (1.75,3)..(2.1,2.5);
\draw[edge] (2.1,2.5)..controls (2,3)..(2.1,3.5);

\draw[edge](1,2.5)..controls (1.75,2)..(2.4,2.5);
\draw[edge](2.4,2.5)--(3.5,2.5);

\draw[edge](1,3.5)..controls (1.75,4)..(2.1,3.5);
\draw[edge](2.1,3.5)..controls (3,4)..(3.5,3.5);
\node[vertex](a) at (1,3.5){X};
\node[vertex](b) at (3.5,3.5){X};

\draw[edge](2.4,2.5)..controls (3,2)..(3.5,2.5);

%\uncover<2>{
%\node[vertex](u2) at (2.4,2.5)[label=below right: \tiny{$x(1)$}][circle,draw=black]{};
%\node[vertex](y) at (2.1,3.5)[label=above: \tiny{$y(2)$}][circle,draw=black]{};
%}

%\uncover<3>{
%\node[vertex](y) at (2.1,3.5)[label=above: \tiny{$y(2)$}][circle,draw=black]{};
%\node[vertex](x) at (2.4,2.5)[label=below right: \tiny{$x(1,2)$}][circle,draw=black]{};
%}

\node[vertex](u2) at (2.4,2.5)[label=below: \tiny{$1$}][circle,draw=black]{};
\node[vertex](y) at (2.1,3.5)[label=above: \tiny{$2$}][circle,draw=black]{};

\node[vertex](u6) at (7.25,2.5)[label=below: \tiny{$4$}][circle,draw=black]{};
\draw[edge](u6)..controls (8,3)..(8.5,2.5);
\draw[edge](6,2.5)..controls (6.5,3)..(u6);

%\node[vertex](u31) at (5.5,2.5)[circle,draw=black]};
\node[vertex](u32) at (6.5,2.5)[label=below:\tiny{$3$}][circle,draw=black]{};

%\node[vertex](u41) at (8,2.5)[circle,draw=black]{};
\node[vertex](u42) at (8.5,3.5)[label=below:\tiny{$4$}][circle,draw=black]{};
\draw[edge] (u42)..controls (9,3)..(8.5,2.5);
\node at (8.5,2.5){X};

\node[vertex](s1) at (1.5,2.5)[]{};

\node[vertex](u11) at (1,2.5)[]{};
\node[vertex](u21) at (4,2.5)[]{};
%\draw[blue edge] (u11)..controls (2.5,3.5) and (6,3.5)..(u41);
%\draw[blue edge] (u21)..controls (5,2)..(u32);

\node[vertex](u5) at (4.75,2.5)[label=below: \tiny{$2$}][circle,draw=black]{};
\draw[edge](u5)..controls (5.25,3)..(6,2.5);

\draw[edge](3.5,2.5)..controls (4.25,3)..(u5);
\node at (6,2.5){X};

\end{tikzpicture}
\caption{Labeling $2$}
\end{figure}
\end{center}

Next we describe \emph{Labeling 3}. For every $h\in H$, if $D^h$ contains vertices $x$ and $y$ both with paired labels $(j_1,j_2)$, for some integers $j_1,j_2,$, then we relabel $x$ by the label $j_1$ and $y$ by the label $j_2$. For every $h\in H$, if $D^h$ contains vertices $x$ and $y$ with paired label $(j_1,j_2), (j_2,j_3)$ respectively, for some integers $j_1,j_2,$ and $j_3$, then we relabel $y$ by the label $j_3$. If $x$ and $y$ are labeled $j_1$ and $(j_1,j_2)$ respectively, for some integers $j_1,j_2$, we relabel $y$ by $j_2$. We apply this relabeling to pairs of vertices of $D^h$, sequentially, in any order. This completes Labeling $3$.
\medskip

%The final labeling we call \emph{Labeling 4}. For every $h\in H$ and distinct $i,j\in [k]-I^h$, if $v\in D^h\cap P_{i,j}$ is labeled by the paired label $(i,j)$, then relabel $v$ arbitrarily by $i$ or $j$. This completes Labeling $4$.
%\medskip

For $h\in H$, let $S_1^h$ be the set of vertices of $D^h$ which still have a pair of labels. Notice that after Labeling $3$, $S_1^h$ is contained in $\C_{I^h}$. For each vertex in $S_1^h$, we place each component of the paired label on that vertex in the set $J^h_1$. For example, if $S_1^h$ contains vertices with labels $(i_1,i_2)$ and $(i_3,i_4)$, then $J^h_1=\{i_1,i_2,i_3,i_4\}$.

Define the index set $I^h_1=[k] - I^h=\{i_{\ell+1}, \dots, i_{k}\}$ for vertically dominated cells of $G^h$. 

%Choose $x\in S_1^h$, in particular, $x\in P_{j_1,j_2}$ for $j_1,j_2\in J^h_1$. If $h'\in N_H(h)$ and there exists a vertex $y\in D^{h'}\cap P^{h'}_{j_1,i}$ for $i\in I^h_1$, then since Labeling $3$ is faithfull and all vertices of $D$ have exactly one label after its completion, we consider two possibilities.

%Case $1$: Suppose that after performing Labeling $3$, $y$ is labeled by $j_1$. This means that after projecting all vertices labeled $j_1$ onto $H$, $h$ is dominated by a dominating vertex labeled $j_1$ at $h'$. We label $x$ by $j_2$.

%Case $2$: Suppose that after performing Labeling $3$, $y$ is labeled by $i$. This means that after projecting all vertices labeled $i$ onto $H$, $h$ is dominated by a dominating vertex labeled $i$ at $h'$. %In this case we say $Q^h_i$ is \emph{covered} and that $y$ is a \emph{covering vertex}. Cells which are not covered are called \emph{uncovered}.

%For $h\in H$, let $S_2^h$ be the vertices of $D^h$ which still have a pair of labels. For each vertex in $S_2^h$, we place each component of the pair of labels on that vertex in the set $J^h_2$.

The following observations follow from the definition of claw-free:

\begin{enumerate}
\item For $j_1,j_2\in [k]-I^h$, no vertex of $D\cap P^h_{j_1,j_2}$ may dominate any of $v_{i_1}^h,\dots,v_{i_{\ell}}^h$. Thus, $\{v_{i_1}^h,\dots, v_{i_\ell}^h\}$ must be dominated horizontally in $G^h$ by shared neighbors of $\{v_i^h:i\in I^h\}$ from the chamber of $Q_{I^h}^h$.
\item If $j_1,j_2,j_3,j_4$ are distinct elements of $[k]$ and $x\in P^h_{j_1,j_2}, y\in P^h_{j_3,j_4}$, then $x$ is not adjacent to $y$.
\item Similarly, $x\in P^h_{j_1}$ is not adjacent to any $y\in P^h_{j_2,j_3}$.
\item By $(2)$, all vertices of $D^h-\C_{J_1^h}$ which are adjacent to some vertex of $\C_{J_1^h}$ must be members of $P^h_{i}$ for $i\in I^h_1$.
\item If a vertex of $\C_{J_1^h}$ is (a) vertically undominated and (b) dominated from outside $\C_{J_1^h}$, then it must be a member of $P^h_j$ for some $j\in J^h_1$, since neither shared neighbors of $\C_{J_1^h}$, nor $v_j^h$ for $j\in J_1^h$, can be adjacent to vertices outside $\C_{J_1^h}$.
\end{enumerate}

Observations $(1)-(5)$ imply the following:

\begin{claim}\label{horizdom}
If $v$ is a vertically undominated vertex of $\C_{J_1^h}$ which is not dominated by a shared neighbor (from $\C_{J_1^h}$ or outside $\C_{J_1^h}$), then it is a private neighbor in $\C_{J_1^h}$. Furthermore, $v$ must be dominated by a private neighbor of $\C_{I_1^h}$.
\end{claim}

%Notice that by the definition of claw-free, for $j_1,j_2\in [k]-I^h$, no vertex of $P^h_{j_1,j_2}$ may dominate any of $v_{i_1}^h,\dots,v_{i_{\ell}}^h$. Thus, $\{v_{i_1}^h,\dots, v_{i_\ell}^h\}$ must be dominated horizontally in $G^h$ by shared neighbors of $\{v_i^h:i\in I^h\}$ from the chamber of $Q_{I^h}^h$. Furthermore, we notice that since $G$ is claw-free,

%\begin{obs}\label{horizprivate}
%If $j_1,j_2,j_3,j_4$ are distinct elements of $[k]$ and $x\in P^h_{j_1,j_2}, y\in P^h_{j_3,j_4}$, then $x$ is not adjacent to $y$. Similarly, $x\in P^h_{j_1}$ is not adjacent to any $y\in P^h_{j_2,j_3}$. 
%\end{obs}

%We note here that by Observation \ref{horizprivate} all vertices of $D^h-\C_{J_1^h}$ which are adjacent to some vertex of $\C_{J_1^h}$ must be members of $P^h_{i}$ for $i\in I^h_1$. Furthermore, if a vertex of $\C_{J_1^h}$ is vertically undominated and dominated from outside $\C_{J_1^h}$, then it must be a member of $P^h_j$ for some $j\in J^h_1$, since neither shared neighbors of $\C_{J_1^h}$, nor $v_j^h$ for $j\in J_1^h$, can be adjacent to vertices outside $\C_{J_1^h}$.

%$(1)$ means that $\{v_{i_1}^h,\dots,v_{i_{\ell}}^h\}$ are dominated by shared neighbors from $\C_{I^h}$.
%$(2)$ means that shared neighbors with distinct indices are independent
%$(3)$ means that private neighbors are independent from shared neighbors if the indices are all distinct
%$(4)$ applies $(2)$ to the graph in question to state that only private neighbors outside $\C_{J_1^h}$ can dominate inside $\C_{J_1^h}$ 
%$(5)$ 

Set $D_{i_j}^h=D^h \cap P_{i_j}^h$ for $\ell+1\leq j \leq k$ and $D^h_{I^h_1}=D^h \cap \C_{I^h_1}$. Let $E_{J^h_1}^h$ be a minimum subset of vertices of $D_{I^h_1}^h$ so that $(D\cap \C_{J^h_1})\cup \left(D\cap N_H(\C_{J^h_1})\right)\cup E_{J^h_1}^h$ dominates $\C_{J^h_1}$. That is, $E_{J^h_1}^h$ is a minimum set of vertices with neighbors in $\C_{J^h_1}$, which along with the dominating vertices in $\C_{J^h_1}$ and $N_H(\C_{J^h_1})$, dominate $\C_{J^h_1}$. However, note that due to Labeling $2$ and the definition of $J_1^h$, $D\cap N_H(\C_{J^h_1})$ is empty. Thus, $E_{J^h_1}^h$ is a minimum subset of vertices of $D_{I^h_1}^h$ so that $(D\cap \C_{J^h_1})\cup E_{J^h_1}^h$ dominates $\C_{J^h_1}$. By Claim \ref{horizdom}, $E_{J^h_1}^h$ is composed of private neighbors of $\C_{I_1^h}$, and hence are labeled by a singleton label.

\begin{claim}\label{big}
For every $h\in H$, $|E^h_{J^h_1}|\geq |S^h_1|$.
\end{claim}

\begin{proof}
Suppose not. Set $j=|E^h_{J^h_1}|$ and $s=|S^h_1|$. Notice that $E^h_{J^h_1}\cup S^h_1$ dominates $\C_{J_1^h}$. Furthermore, since after Labeling $3$ label pairs are disjoint, $|J_1^h|=2s$. Note that $E^h_{J_1^h}$ may contain vertices that are shared neighbors of $\Gamma^h$, which were relabeled in Labeling $3$ to singleton labels. To address this, we define the set of labels of vertices in $E^h_{J^h_1}$ which had been reduced from paired labels to singleton labels as $L^h$. If we let $I'=[k]-J_1^h-L^h$, then $E^h_{J^h_1}\cup S^h_1\cup (\bigcup_{i\in I'}v_i^h)$ dominates $G^h$. However, such a set contains at most $j+s+k-2s=j-s+k<k$ vertices, which contradicts the minimality of $\gamma(G)$.
\end{proof}

%%%%%%%%%%%%%%%%%%%%%%%%%%%%%%%%%%%%

By Claim \ref{big}, $D^h$ contains $|S^h_1|$ vertices labeled by a paired label and at least as many vertices labeled by a singleton label.

%at most $\lfloor\frac{\ell}{2}\rfloor$ vertices which are each labeled by two labels. Call the set of these vertices $C_1^h$. In $D^h$, there are at least $\lceil\frac{\ell}{2}\rceil$ vertices, each labeled by a single label. Call the set of these vertices $C_2^h$.

\begin{claim}
For a fixed $i$, $1\leq i \leq k$, projecting all vertices such that $i$ is an element of the label (singleton or in a pair) to $H$ produces a dominating set of $H$. 
\end{claim}
\begin{proof}
For a fixed $i\in [k]$, if $Q^k_i$ is not vertically dominated, then $D^h$ must contain a vertex adjacent to $v^h_i$. Such a vertex either contains $i$ in its label or, after Labeling $2$, there exists $h'\in V(H)$ adjacent to $h$, and $u^{h'}\in D^{h'}$ so that $u^{h'}$ has $i$ in its label. In either case, projecting vertices of $D$ with $i$ as an element of their labels onto $H$ produces a dominating set of $H$.
\end{proof}

Call the set of such vertices of $D$ labeled $i$, $D_i$. Summing over all $i$ we count at least $\gamma(G)\gamma(H)$ vertices of $D$ where we count the members of $S^h_1$ twice and the members of $E^h_{J^h_1}$ and $D^h-S^h_1-E^h_{J^h_1}$ once, for every $h\in H$. For a fixed sum $\sum_{i=1}^k |D_i|$, $|D^h|$ is minimized when we maximize the number of dominating vertices that are counted twice. Thus we obtain,

\begin{align*}
&\gamma(G)\gamma(H) \leq \sum_{i=1}^k|D_i|\leq 2\frac{|D|}{2}+\frac{|D|}{2}=\frac{3}{2}|D|
\end{align*}

\end{proof}

 \bibliographystyle{plain}
 
 \end{document}